\documentclass[12pt]{amsart}
\usepackage{setspace}
\usepackage{amsfonts}
\usepackage{amsmath}
\usepackage{amssymb}
\usepackage{amsthm}
\usepackage{mathrsfs}
\usepackage[all]{xy}

\newcommand{\proofend}{\hfill \hbox{\vrule width 5pt height 5pt depth
0pt}}

\newcommand{\C}{\mathbb{C}}
\newcommand{\Z}{\mathbb{Z}}
\newcommand{\N}{\mathbb{N}}
\newcommand{\Q}{\mathbb{Q}}

\newcommand{\Oz}{\mathcal{O}}

\newcommand{\proj}{\mathbb{P}}
\newcommand{\A}{\mathbb{A}}

\newcommand{\F}{\mathbb{F}}

\newcommand{\spec}{\mathrm{Spec} \,}

\newtheorem{thm}{Theorem}
\newtheorem{lemma}{Lemma}[section]
\newtheorem{corol}[thm]{Corollary}
\newtheorem{conj}[thm]{Conjecture}

\begin{document}

\title[Finiteness of local torsion]{Finiteness of local torsion \\ for abelian $t$-modules}
\author{Vesselin Dimitrov}
\address{Yale University Math. Dept. \\ 10 Hillhouse Avenue \\ CT~06520--8283 }
\email{vesselin.dimitrov@yale.edu}

\begin{abstract}
Let $C/\F_q$ be a regular projective curve, $\infty \in C$ a closed point,  $A := \Gamma(C - \{\infty\}, \Oz_C)$, and $K := K(C)$ the fraction field of $A$.  Consider a finite extension $L/K$, a place $v$ of $L$, and an abelian $A$-module $M$ (in the sense of Anderson) over $L_v$. We prove that the $L_v$-rational torsion submodule $M(L_v)_{\mathrm{tors}}$ of $M$ is a finite $A$-module.
 \end{abstract}

\maketitle

\section{Introduction} \label{intro}

\subsection{}  \label{fak} The present note is motivated by a conjecture in algebraic dynamics over local fields of positive characteristic. It arose as a question in Fakhruddin's  paper~\cite{fakhruddin}, where among other results it was proven that a morphism $\varphi : \proj_{\Z_p}^N \to \proj_{\Z_p}^N$ of degree at least two possesses only finitely many $\Q_p$-rational preperiodic points. (It is crucial here that $\varphi$ is defined over $\Z_p$ and not merely over $\Q_p$, as is evidenced by the iteration $z \mapsto (z^p-z)/p$ of $\proj_{\Q_p}^1$, all of whose preperiodic points lie in $\Z_p \cup \{\infty\}$.) The following $\F_q[[T]]$-counterpart of this statement remains open.

\begin{conj}  \label{fakconj}
Let $\varphi : \proj_{\F_q[[T]]}^N \to \proj_{\F_q[[T]]}^N$ be a morphism of degree at least two. Then, $\varphi$ has only finitely many $\F_q((T))$-valued preperiodic points.
\end{conj}

It is shown in~\cite{fakhruddin} that there is an $n > 0$ such that all periods of $\F_q((T))$-valued points for $\varphi^n$ are powers of the equicharacteristic $p$. See also Zieve's thesis~\cite{zieve}, Theorem~50, for a more precise result in the case $N = 1$.

\subsection{}  \label{abelianvar} An evidence for Conjecture~\ref{fakconj} is the well known  fact that an abelian variety $A$ over any non-archimedean local field $F$ has finite torsion group: $\# A(F)_{\mathrm{tors}} < \infty$. See, for example, Clark and Xarles~\cite{clarkxarles}. Here, $F$ can be a $p$-adic field $\Q_p$ as well as a Laurent series field $\F_q((T))$; hence, the analog of Conjecture~\ref{fakconj} holds for an isogeny $\varphi : A \to A$ of degree at least two.

\subsection{}  In function field arithmetic the counterpart of an abelian variety is the notion of an abelian $t$-module, introduced by Anderson in~\cite{anderson}. It includes, as the special case of dimension one, the notion of a Drinfeld module. The isogenies of such modules over $\F_q((T))$ give rise to particularly well behaved additive endomorphisms of affine space $\A_{\F_q((T))}^N$; many of them extend as endomorphisms of $\proj_{\F_q((T))}^N$.  The class of additive polynomials is of particular interest in Conjecture~\ref{fakconj}. We show that no counterexamples arise from an abelian $t$-module, by proving the analog of the finiteness result about abelian varieties in~\ref{abelianvar}.

\subsection{} \label{termin} To state our result, we need to introduce some notation. Let $C/\F_q$ be a regular, projective, geometrically connected curve, and consider a closed point $\infty \in C$, declared as the place at infinity. Consider the ring $A := \Gamma(C - \{\infty\}, \Oz_C)$ of functions on $C$ regular outside $\infty$, and $K$ its quotient field (the function field of $C$). Let further $L/K$ be a finite extension, $v$ a place of $L$, and $L_v$ the $v$-adic completion. An \emph{Anderson $A$-module} of dimension $d$ over $L_v$ is a homomorphism $\Phi : A \to \mathrm{End}(\mathbb{G}_{a/L_v}^d)$ such that, for every $f \in A$, the derivative (zeroth order term) $D\Phi(f) \in \mathrm{End}(\mathrm{Lie}(\mathbb{G}_{a/L_v}^d)) \cong M_d(L_v)$ has all its eigenvalues equal to $f$.

Let $\tau : x \mapsto x^q$ be the Frobenius substitution and $L_v\{\tau\}$ the twisted Frobenius ring defined by the commutation relation $\tau c = c^q \tau$, for $c \in L_v$. An $A$-module $\Phi$ over $L_v$ endows $\mathrm{Hom}(\mathbb{G}_{a/L_v}^d,\mathbb{G}_{a/L_v})$ with an $A \otimes_{\F_q} L_v-L_v\{\tau\}$ bimodule structure in an obvious way. Then, following Anderson~\cite{anderson}, $\Phi$ is said to be an \emph{abelian $A$-module} if  $\mathrm{Hom}(\mathbb{G}_{a/L_v}^d,\mathbb{G}_{a/L_v})$ is free of finite rank as  $A \otimes_{\F_q} L_v$-module. The latter condition is equivalent to the finite generation of the $A \otimes_{\F_q} L_v$-module  $\mathrm{Hom}(\mathbb{G}_{a/L_v}^d,\mathbb{G}_{a/L_v})$ (see~\cite{anderson}, Lemma~1.4.5).

\begin{thm} \label{main}
Let $\Phi$ be an abelian $A$-module over $L_v$. Then, the $L_v$-rational torsion module
$$
\{ x \in \mathbb{G}_a^d(L_v) \mid \Phi(f)(x) = 0 \textrm{ for some } f \in A \setminus \{0\} \}
$$
is a finite $A$-module.
\end{thm}

The paper is organized as follows. In section~\ref{julia} we consider the filled Julia $A$-module of $\Phi$ and prove that it is bounded. It is here that the hypothesis that $\Phi$ is abelian is used in a crucial way; we do not know if the statement of our theorem continues to hold for arbitrary Anderson $A$-modules. In section~\ref{formal} we consider the formal module attached to an Anderson $A$-module over the valuation ring of $L_v$ and exploit the positive radius of convergence of the exponential map. The proof of Theorem~\ref{main} is then completed in section~\ref{fine}.

\bigskip

\section{Boundedness of the filled Julia module} \label{julia}

Let $\Phi$ be an Anderson $A$-module over $L_v$. Thus $\mathbb{G}_a^d(L_v)$  acquires an $A$-module structure via $\Phi$. We follow Ingram~\cite{ingram} in defining the \emph{$L_v$-rational filled Julia module} $\Phi^{\mathrm{FJ}}(L_v)$ to be the $A$-submodule consisting of the $x \in \mathbb{G}_a^d(L_v)$ having $v$-adically bounded $A$-orbit. Ingram works with the case $d = 1$ of Drinfeld modules, and he extends the definition of $\Phi^{\mathrm{FJ}}$ to the Berkovich affine line to consider the connected component of the identity. 

For our purpose of proving finiteness of local torsion, the consideration of the $L_v$-rational part of $\Phi^{\mathrm{FJ}}$ is sufficient. The following higher dimensional result is also due to Ingram and is essentially contained in the unpublished manuscript~\cite{ingraman}. The short proof presented here was found independently by the present author, following the idea of the one dimensional case presented in~\cite{ingram}.

\begin{lemma}[Ingram] \label{bounded}
If $\Phi$ is abelian, the set $\Phi^{\mathrm{FJ}}(L_v)$ is $v$-adically bounded.
\end{lemma}

\begin{proof}
Let $m_1, \ldots, m_d : \mathbb{G}_a^d \to \mathbb{G}_a$ be the coordinate projections. They form a free basis of $M := \mathrm{Hom}(\mathbb{G}_{a/L_v}^d,\mathbb{G}_{a/L_v})$ as $L_v\{\tau\}$-module. Our assumption is that $M$ is also finitely generated as $A \otimes_{\F_q} L_v$-module; let $\alpha_1, \ldots, \alpha_r$ be a set of generators.

Write also $v : L_v \setminus \{0\} \twoheadrightarrow \Z$ for the normalized valuation of the local field $L_v$. We claim that for all $n \in \N$ and $i$ there is a presentation
\begin{equation} \label{present}
\tau^n  m_i = \sum_{j=1}^r (f_{ijn} \otimes c_{ijn}) \alpha_j, \quad v(c_{ijn}) \geq -Cq^n,
\end{equation}
where $C < \infty$ is a constant independent of $n$.

The claim implies our desired conclusion. Indeed, if $v(m_i(x)) < -2C$ for some $i$ then $v(\tau^n m_i(x)) = q^n v(m_i(x)) < -2Cq^n$, and~(\ref{present}) as $n \to +\infty$ shows that $x$ cannot lie in the filled Julia set: otherwise it would hold $v((f_{ijn} \otimes c_{ijn}) \alpha_j (x)) = v(c_{ijn}\alpha_j(\Phi_{f_{ijn}}(x)) ) \geq v(c_{ijn}) + O_n(1)$ for all~$j$.

To prove our claim, note that $\tau \cdot (f \otimes c) \alpha = (f \otimes c^q )  \tau \alpha$ with $v(c^q) = q v(c)$. Since by assumption the $m_i$ and the $\tau \alpha_j$ can be expressed as $A \otimes_{\F_q} L_v$-linear combinations in $\alpha_1, \ldots, \alpha_r$, a presentation~(\ref{present}) follows by induction on~$n$.
\end{proof}

\medskip

{\it Remark. }  It is essential in Lemma~\ref{bounded} that $\Phi$ be abelian; the trivial $t$-module given by $\Phi_t : \, x \mapsto tx$ over $\F_q(t)$ has $\Phi^{\mathrm{FJ}}(L_v) = L_v$ for all $v$. This  reaffirms that abelian $t$-modules are the correct analog of abelian varieties. In this analogy, the basic boundedness result verified here corresponds to the existence of the N\'eron model. This parallel was detailed by Ingram in~\cite{ingram} for the case of dimension one (Drinfeld modules). It would be interesting more generally to study the structure of the component $A$-module of the filled Julia set of $\Phi$ on the Berkovich affine space $\A_{/\C_v}^{d,\mathrm{an}}$. (The field $\C_v$  in the notation here is, as usual, the completion of an algebraic closure of $L_v$.)

\bigskip

\section{The formal module}  \label{formal}

Let $O_v$ be the valuation ring of $L_v$ and $\mathfrak{m}_v$ its maximal ideal. In this section we assume $\Phi$ is an Anderson $A$-module over $L_v$ such that
\begin{equation} \label{uni}
\Phi(f) \equiv f 1_d + N_f \mod{   \mathfrak{m}_v M_d(O_v\{\tau\})  }
\end{equation}
for all $f \in A$. Here, as required by the definition of an Anderson module, $N_f \in M_d(L_v)$ is a nilpotent matrix. Any $A$-module over $L_v$ can be brought to such a form by conjugation with a scalar.

For $f \notin \mathfrak{m}_v \cap A$, the element $\Phi(f) \in M_d(L_v\{\tau\})$ satisfying~(\ref{uni}) can be inverted in $M_d(L_v\{\tau\}) + \mathfrak{m}_v M_d(O_v\{\{\tau\}\}) \subset M_d(L_v\{\{\tau\}\})$. In this way, letting $R := A_{\mathfrak{m}_v \cap A}$ the local ring of $\spec{A}$ at $\mathfrak{m}_v \cap A$, we obtain a formal Anderson $R$-module
$$
\widehat{\Phi} : R \to M_d(L_v\{\{\tau\}\})
$$
over $L_v$, such that for all $r \in R$, $\widehat{\Phi}(r)(x)$ converges in $\mathbb{G}_a^d(L_v)$ for all $x \in \mathbb{G}_a^d(\mathfrak{m}_v)$.

\subsection{The formal logarithm} The proof of Prop.~2.1 in Rosen~\cite{rosen} (which treats the case $d = 1$) extends almost verbatim to yield  a unique $l_{\Phi} \in M_d(L_v\{\{\tau\}\})$ satisfying $Dl_{\widehat{\Phi}} = 1_d$ and $l_{\widehat{\Phi}} \circ \widehat{\Phi}(r) = r l_{\widehat{\Phi}}$ for all $r \in R$.

We indicate the necessary changes. As in {\it loc. cit.}, given the first identity $Dl_{\widehat{\Phi}} = 1_d$, it is enough to verify the second identity for $r = \pi$, a fixed generator for the maximal ideal of $R$. Write $\widehat{\Phi}(\pi) = \pi 1_d + N_{\pi} + \sum_{n \geq 1} B_n \tau^n$ and $l = 1_d + \sum_{n \geq 1} C_n \tau^n$ with $N_{\pi}$ nilpotent and undetermined coefficients $C_n \in M_d(L_v)$; here, almost all $B_n$ are in $M_d(O_v)$. Then $l_{\widehat{\Phi}} \circ \widehat{\Phi}(\pi) = \pi l_{\widehat{\Phi}}$ translates into the recursion
\begin{equation} \label{reclog}
\big( (\pi - \pi^{q^n}) 1_d - N_{\pi}^{q^n}\big)C_n = B_n + \sum_{\substack{i+j = n \\ j < n}} C_jB_i^{q^j}
\end{equation}
Since $N_{\pi}$ is nilpotent, the coefficient on the left-hand side of~(\ref{reclog}) is invertible for all $n \geq 1$, and equals the scalar matrix $(\pi - \pi^{q^n})1_d$ for $n \gg 0$. As in~\cite{rosen}, we see from this description that~(\ref{reclog}) defines uniquely the desired $l_{\widehat{\Phi}}$, and that $l_{\widehat{\Phi}}(x)$ converges for $x \in \mathbb{G}_a^d(\mathfrak{m}_v)$. Moreover, writing as before $m_i : \mathbb{G}_a^d \to \mathbb{G}_a$ the coordinate projections and $v : O_v \twoheadrightarrow \Z$ the normalized valuation, we obtain:

\begin{lemma} \label{loga}
There is a $k_1 = k_1(\Phi) < \infty$ such that, for $x \in \mathbb{G}_a^d(\mathfrak{m}_v^{k_1})$, it holds $v(m_i(l_{\widehat{\Phi}}(x))) = v(m_i(x))$, for $i = 1,\ldots,d$. \proofend
\end{lemma}

\subsection{The exponential map} This is the inverse
$$
e_{\widehat{\Phi}} := l_{\widehat{\Phi}}^{-1} \in M_d(L_v \{\{\tau \}\})
$$
 of the formal logarithm. It satisfies $e_{\widehat{\Phi}} r = \widehat{\Phi}(r)e_{\widehat{\Phi}}$ for all $r \in R$.

 \begin{lemma} \label{expo}
There is a $k_2 = k_2(\Phi)$ such that $e_{\widehat{\Phi}}$ converges on $\mathbb{G}_a^d(\mathfrak{m}_v^{k_2})$, and has $v(m_i(e_{\widehat{\Phi}}(x))) = v(m_i(x))$, $i = 1,\ldots,d$, for all $x \in \mathbb{G}_a^d(\mathfrak{m}_v^{k_2})$.
 \end{lemma}

\begin{proof}
This follows from the same calculations as in Propositions~2.2 and 2.3 of Rosen~\cite{rosen} (who considers the case $d = 1$). In fact, $
e_{\widehat{\Phi}}$ converges for all $x$ having $v(m_i(x)) > e/(q-1)$ for all $i = 1,\ldots,d$, where $e$ is the ramification index of $O_v/R$.
\end{proof}

\subsection{Conclusion}  Combining Lemmas~\ref{loga} and~\ref{expo}, let
$$
k := \max(k_1,k_2) < \infty.
$$
 We obtain:

\begin{corol}  \label{keypt}
If $x \in \mathbb{G}_a^d(\mathfrak{m}_v^k)$ has $\Phi(f)(x) = 0$ for some $f \in A \setminus \{0\}$, then $x = 0$.
\end{corol}

\begin{proof}
We also have $\widehat{\Phi}(f)(x) = 0$. Evaluating $e_{\widehat{\Phi}} r = \widehat{\Phi}(r)e_{\widehat{\Phi}}$ at $l_{\widehat{\Phi}}(x)$ we obtain $e_{\widehat{\Phi}}( f l_{\widehat{\Phi}}(x) ) = 0$. Since $f \neq 0$, the two lemmas then yield $x = 0$.
\end{proof}

\section{Finiteness of local rational torsion} \label{fine}

Combining Lemma~\ref{bounded} and Corollary~\ref{keypt} we can now complete the proof of Theorem~\ref{main}. First of all, as noted at the beginning of section~\ref{formal}, conjugating $\Phi$ by a scalar $\lambda$ with $v(\lambda) \gg 0$ we may assume, for all $f \in A$, that $\Phi(f)$ satisfies condition~(\ref{uni}). (The $j$-th matrix coefficient of $\lambda^{-1}\Phi \lambda$ is $\lambda^{d(q^j - 1)}$ times the corresponding coefficient of $\Phi$.) Then Corollary~\ref{keypt} applies, for some $k = k(\Phi) < \infty$. Since the $A$-modules $\Phi$ and $\lambda^{-1} \Phi \lambda$ are isomorphic over $L_v$, we lose no generality in making this reduction.

Let us now fix a non-constant element $f \in A$. By Lemma~\ref{bounded} and the pigeonhole principle, there is an $M = M(k) < \infty$ such that, for all $x \in \Phi^{\mathrm{FJ}}(L_v)$, there are $0 \leq j_1 < j_2 \leq M$ such that $m_i(\Phi(f^{j_1})(x)) \equiv m_i(\Phi(f^{j_2})(x)) \mod{\mathfrak{m}_v^k}$ for $i = 1,\ldots,d$; this means $\Phi(f^{j_1} - f^{j_2})(x) \in \mathbb{G}_a^d(\mathfrak{m}_v^k)$. Since each $L_v$-rational torsion point is in $\Phi^{\mathrm{FJ}}(L_v)$ by the definition of the $L_v$-rational filled Julia module, this together with Corollary~\ref{keypt} yields a finite set $F \subset A$ of non-zero elements of $A$ such that every $L_v$-rational torsion point for $\Phi$ has order belonging to $F$. Altogether, this yields finiteness for the $L_v$-rational torsion.  \proofend

\section{Acknowledgements}

At the time of writing of this note I was unaware of Ingram's unpublished manuscript~\cite{ingraman}, where the compactness of the adelic filled Julia module of an abelian $A$-module $\phi$ over a global function field had been proved and used to derive that the torsion module of $\phi$ consists of points of bounded height. In any case the short proof given here of Lemma~\ref{bounded} follows the idea in Ingram's earlier paper~\cite{ingram} treating the case of a Drinfeld module.
It is a pleasure to thank Patrick Ingram for sending me his manuscript~\cite{ingraman}, and Najmuddin Fakhruddin, Dragos Ghioca and Patrick Ingram for several enlightening remarks about the problem considered here.


\begin{thebibliography}{99}

\bibitem{anderson} Anderson G.: $t$-Motives, {\it Duke Math. J.}, vol. {\bf 53}, no. 2 (1986), pp. 457--502.



\bibitem{clarkxarles} Clark P., X. Xarles: Local bounds for torsion points on abelian varieties, {\it Canad. J. Math.}, vol. {\bf 60} (2008), pp. 532--555.


\bibitem{fakhruddin} Fakhruddin N.: Boundedness results for periodic points on algebraic varieties, {\it Proc. Indian Acad. Sci. - Math. Sci.}, vol. {\bf 111}, no. 2 (2001), pp. 173--178.

\bibitem{ingram} Ingram P.: The filled Julia set of a Drinfeld module and uniform bounds for torsion (preprint, 2012),  \texttt{arXiv:1210.3059v2}.

\bibitem{ingraman} Ingram P.: The filled Julia set and torsion for abelian $t$-modules (preprint, unpublished).

\bibitem{rosen} Rosen M.: Formal Drinfeld modules, {\it J. Number Theory}, vol. {\bf 103} (2003), pp. 234--256.



\bibitem{zieve} Zieve M.: Cycles of Polynomial Mappings, PhD thesis, University of California Berkeley (1996).

\end{thebibliography}
\end{document}